\documentclass[11pt]{article}

\usepackage{amsmath,amssymb,amsthm,color,graphicx}

\newtheorem{theorem}{Theorem}[section]
\newtheorem{lemma}{Lemma}[section]
\newtheorem{claim}{Claim}[section]

\newtheorem{observation}{Observation}[section]
\newtheorem{property}{Property}[section]

\newcommand{\sumodd}{\oplus}
\newcommand{\Sumodd}{\bigoplus}

\newcommand{\comment}[1]{}
\newcommand{\ignore}[1]{}

\def\R{{\mathbb R}}
\def\N{{\mathbb N}}
\def\Z{{\mathbb Z}}
\def\ar{{\it r}}
\def\half{{{1 \over 2}}}
\newcommand{\circplus}{\,\mbox{\rm \r{+}}\,}

\def\clap#1{\hbox to 0pt{\hss#1\hss}}

\makeatletter
  \def\moverlay{\mathpalette\mov@rlay}
  \def\mov@rlay#1#2{\leavevmode\vtop{%
    \baselineskip\z@skip \lineskiplimit-\maxdimen
    \ialign{\hfil$#1##$\hfil\cr#2\crcr}}}
\makeatother

\newcommand\A{\mathcal{A}}

\def\T{{\cal T}}
\def\P{{P}}

\def\A{{\cal A}}
\def\S{{X}}
\def\F{{\cal F}}
\def\Fgoth{{\mathfrak F}}
\def\fgoth{{\mathfrak f}}

\newcommand{\remove}[1]{}

\begin{document}
\title{Rational Polygons: Odd Compression Ratio and Odd Plane 
Coverings\footnote{The final version of the contribution is due to be published in the collection of papers 
{\em A Journey through Discrete Mathematics:\! A Tribute to Jiri Matousek} edited by Martin Loebl, Jaroslav Nesetril and Robin Thomas, due to be published by Springer.}}
\author{
Rom Pinchasi\thanks{
Math.~Department,
Technion--Israel Institute of Technology,
Haifa 32000, Israel.
{\tt room@math.technion.ac.il}.
Supported by ISF grant (grant No. 409
/16)}
\and
Yuri Rabinovich\thanks{
Department of Computer Science, Haifa University, Haifa, Israel.
{\tt yuri@cs.haifa.ac.il}.
}
}
\maketitle
\begin{abstract}
Let $P$ be a polygon with rational vertices in the plane.
We show that for any finite odd-sized collection of translates of $P$,
the area of the set of points lying in an odd number of these
translates is bounded away from $0$ by a constant
depending on $P$ alone.

The key ingredient of the proof is a construction of an odd
cover of the plane by translates of $P$. That is, 
we establish a family $\F$ of translates of $P$ covering (almost) every 
point in the plane a uniformly bounded odd number of times. 
\end{abstract}

\section{Introduction}
The starting point of this research is the following isoperimetric-type problem
about translates of compact sets in $\R^d$:

{\em Let $\S\subset \R^d$ be a compact set, and let $Z \subset \mathbb{R}^d$ be a finite 
set of odd cardinality. Consider the finite odd-sized collection ${\cal F} = \{\S + z\}_{z\in Z}$ of translates of $\S$. Let $U\subset \R$ be the set of all points that belong to an odd number of the members of $\cal F$. How small can be the Lebesgue measure of $U$ in terms
of the Euclidean measure of $\S$?}

Denoting the infimum of this value by $V\!ol_{odd}(\S)$, called
the {\em odd volume} of $\S$, we define the {\em 
odd compression ratio} of $\S$ as $\alpha^\circ(\S) = V\!ol_{odd}(\S)\,/\,V\!ol(\S)$,
where $V\!ol(\S)$ is the Euclidean volume of $\S$.
Observe that $\alpha^\circ(\S) \leq 1$, as $\cal F$
may consist of a single element $\S$.  Clearly, $\alpha^\circ(\S)$ is an affine invariant. 

It was observed by the second author about a decade ago that $\alpha^\circ$ of a unit $d$-cube $Q^d$
is $1$. Indeed, informally, consider $\R^d$ under the action (i.e., translation) of $\Z^d$. 
The unit cube (with parts of its boundary removed) is a fundamental domain of $\R^d/\Z^d$. 
The quotient map $\phi:\R^d \rightarrow Q^d$ maps any translate of $Q^d$ onto $Q^d$ in a  
one-to-one manner. Moreover, the quotient map satisfies
\[
\phi\left(\Sumodd_{\T \in \cal F} T\right) ~=~  \Sumodd_{\T \in \cal F} \phi(T)\,,
\]
where $\Sumodd$ denotes the set-theoretic union modulo $2$, i.e., 
the set of all points covered by an odd number of the members of $\cal F$. Since the quotient map is
locally volume preserving, it is globally volume-nondecreasing, and so one concludes that the volume of 
$\Sumodd_{\T \in \cal F}\, T$ is at least that of 
$\phi(\Sumodd_{\T \in \cal F}\, T) \;=\;
\Sumodd_{\T \in \cal F}\, \phi(T) \;=\;
\Sumodd_{\T \in \cal F}\, Q^d \;=\; Q^d$\,.

A similar argument shows that $\alpha^\circ$ of a centrally symmetric planar hexagon is
$1$ as well. But what about other sets, i.e., a triangle? The second author vividly 
remembers discussing this question with Jirka Matou\v{s}ek in a pleasant cafe at Mal{\`a} 
Strana, laughing that they are too old for Olympiad-type problems...\footnote{A discrete version
of the problem about the translates of a square in $\R^2$ had indeed found its way 
into a mathematical olympiad~\cite{TT09}. }     

The value of $\alpha^0$ of the triangle (recall that any two triangles are affinely equivalent) was determined 
by the first author in~\cite{P14}; it is ${1\over 2}$. 

Next significant progress on the problem was obtained in~\cite{OPP15}. It was shown
there that for a union of two disjoint intervals of length $1$ on a line with a certain irrational 
distance between them, the odd compression ratio is $0$. The proof uses some algebra of polynomials, 
and Diophantine approximation. The construction easily extends to higher dimensions. 
In addition,~\cite{OPP15} introduced a technique for obtaining lower bounds 
on $\alpha^\circ(\S)$, and used it to show that for $\S$'s that are unions of 
finitely many cells of the $2$-dimensional grid, $\alpha^\circ(\S) > 0$.

In the present paper we further develop the technique of~\cite{OPP15}, and use
it to prove that for any planar rational polygon $P$, the odd compression 
ratio $\alpha^\circ(P)$ is bounded away from $0$ by some positive constant explicitly defined 
in terms of $P$. In fact, the statement applies to any compact planar figure 
with piecewise linear boundary, and (finitely many) rational vertices. 
In view of the above mentioned result from~\cite{OPP15}, the assumption of rationality 
cannot in general be dropped. 

Perhaps more importantly, the value of $\alpha^\circ(\S)$ is related here
to the value of some other natural geometric invariant of $\S$.
The other invariant is $\theta^\circ(\S)$, the smallest possible {\em average density} in an 
{\em odd cover} of $\R^2$ by a family $\F$ of translates of $\S$. By odd cover we mean 
that every point $p\in\R^2$, with a possible exception of a measure $0$ set, is covered 
by the members of $\F$ an odd and uniformly bounded number of times. 

While~\cite{OPP15} does not directly consider odd covers of $\R^2$, it still implies that 
$\alpha^\circ(\S) \geq \theta^\circ(\S)^{-1}$. We include here two complete 
proofs of this useful inequality.
 
The existence of odd covers of $\R^2$ by translates of a rational polygon is by no means
obvious. Most of the present paper is dedicated to constructing such covers.
We are aware of no related results in the literature. 

While many of the results and constructions presented here can be easily extended 
to higher dimensions, some essential parts resist simple generalization, and more work
is required in order to understand the situation there.  

To conclude the Introduction, we hope that the present paper will somewhat elucidate
the meaning of the odd compression ratio $\alpha^\circ(X)$, and that the odd covers 
introduced here will prove worthy of further study. 

\ignore{
{\em Importantly}, to simplify the presentation, throughout this paper we shall not distinguish between sets (or unions of sets) that are identical up to symmetric difference of measure $0$. 
In particular, speaking of unions of polygons, we shall not  care about 
the union of their 1-dimensional boundaries. 
In some (but not all) places we will be explicit about this slight abuse of terminology. 
}
%
\section{Preliminaries}
\subsection{Two Basic Operators}
In what follows, we shall extensively use the following two operators on subsets of $\mathbb{R}^2$: ~$\sumodd$ and $\circplus$. Let us briefly discuss them here. 

The first operator $\sumodd$ is the set-theoretic union modulo $2$. Given a family $\cal F$ of subsets of $\R^2$ so that any $p\in \R^2$ is covered at most
finitely many times by $\cal F$,  ~$\Sumodd_{X \in {\cal F}} X$ is the set of all points of $\R^2$ covered by an odd number of members in $\F$. Observe that $\sumodd$ is commutative and associative.

The second operator, $\circplus$, is less standard. It is the Minkowski sum
modulo $2$:
\[
X \circplus Z ~=~ \Sumodd_{x \in X,\, z\in Z} x + z 
~=~ \Sumodd_{z\in Z} (X + z)\,,
\]
where $X+z$ denotes the translate of $X$ by $z$. I.e., $a \in X \circplus Z$ if and only if the number
of representations of $a$ of the form $\,a=x+z\,$ is an odd natural number.
Unlike the Minkowski sum, $X \circplus Z$ is well defined only when every 
$a \in \R^2$ has at most  finitely many representations of the form $\,x+z\,$ as above. 
This requirement is met, e.g., when $Z$ is finite, or when $Z$ is a discrete set of points at distance $\geq \epsilon >0$ from each other, and $\S$ is bounded.
Since the Minkowski sum extends to any finite number of sets, and it is commutative and associative, the same holds for $\circplus$ (provided, as before, that every $a$ has finitely many representations).

Moreover, the following distributive law holds.
Let ${\cal G}$ be a family of sets in $\R^2$, and let $S \subset \R^2$.
Assume that the family of sets $\{Y+s\}_{Y\in {\cal G},\, s \in S}$\, covers
any point of $\R^2$ at most finitely many times. Then: 
\begin{equation}
\label{eq:distrib}
\left(\Sumodd_{Y \in {\cal G}} Y \right)  \circplus\; {S}  ~=~ \Sumodd_{Y \in {\cal G}} (Y \circplus {S})\,.
\end{equation}
Indeed, the equality is trivial when $S$ consists of a single element.
Thus, by definition of $\circplus$, 
\[
\left(\Sumodd_{Y \in {\cal G}} Y \right)  \,\circplus\, S ~\stackrel{}{=} ~
\Sumodd_{s\in S} \left(\left(\Sumodd_{Y \in {\cal G}}Y\right) + s \right) ~\stackrel{}{=} ~
\Sumodd_{s\in S} \Sumodd_{Y \in {\cal G}} (Y + s) ~\stackrel{*}{=} ~ 
\Sumodd_{Y \in {\cal G}} \Sumodd_{s\in S} (Y + s) ~\stackrel{}{=} ~ \Sumodd_{Y \in {\cal G}} (Y \circplus S)\,.
\]
It remains to validate the change of order of summation in the starred equality.
For $a\in \R^2$ consider the set 
$\{(Y,s)\,|\, a\in Y+s\} \subseteq {\cal G} \times S$.
By our assumptions, this set is always finite. Therefore, for any $a$,
~${\bf 1}_a  \big(\Sumodd_{s\in S} \Sumodd_{Y \in {\cal G}} (Y + s)\big) \;=\;
\Sumodd_{s\in S} \Sumodd_{Y \in {\cal G}} {\bf 1}_a  (Y + s)$ has
only finitely many nonzero terms. Hence, the order of summation in the double sum 
~$\Sumodd_{s\in S} \Sumodd_{Y \in {\cal G}} (Y + s)$~ is interchangeable. 

Finally, notice that similarly to Minkowski sum,  $X \circplus \emptyset = \emptyset$,
while $X \sumodd \emptyset = X$.
\subsection{Covers and Their Densities}
\label{sec:dense}
{\em It is important to stress that throughout this paper whenever
we speak on covers or odd covers of the plane it always means covering {\bf up to a set
of measure $0$}, even if it is not explicitly said so. This convention helps to avoid discussing
unnecessary technicalities related to the boundaries of the sets in the cover.}

For every compact measurable set $X\subset \R^2$, we denote by $A(X)$ the Euclidean area 
of $X$. 
Let $Z \subseteq \R^2$ be a discrete set. The family 
${\cal F} = \{X + z\}_{z\in Z}$ has a {\em uniformly bounded degree} if there exists a constant 
$d_{\cal F}$ such that every ${a\in \R^2}$ belongs to at most $d_{\cal F}$ members of $\cal F$.
Further, such $\cal F$ is called a {\em cover} of $\R^2$ if $X + Z = \R^2$.
I.e., the cover degree of any $a \in \R^2$ by the members of a cover $\cal F$ is uniformly bounded, and, 
up to a set of measure 0, it is strictly positive.  

The (lower) {\em density} of $\cal F$ with a uniformly bounded degree, $\rho({\cal F})$, is defined by 
\[
\rho({\cal F}) ~=~ \liminf_{n\rightarrow \infty}\,  
{{\sum_{z\in Z} A(\, Q_n\cap (X +z))} \over {n^2}},
\]
where $Q_n$ is the $n \times n$ square centered at the origin.  Clearly, $\rho({\cal F}) \geq 1$
when $\cal F$ is a cover or $\R^2$. 

Since ${{\sum_{z\in Z} A(\, Q_n\cap X +z)} /{n^2}}$ is precisely the average
of the cover degrees $d_{\cal F}(a)$ where $a$ ranges over $Q_n$,
the density $\rho({\cal F})$ can be viewed as a kind of an average degree of the cover of $\R^2$ by $\cal F$.   

\begin{claim}
\label{eq:dense}
Fixing $Z$ and varying the (measurable) $X$, the density of the family ${\cal F}=\{X + z\}_{z\in Z}$ is 
proportional to $A(X)$. I.e., $~\rho({\cal F}) ~=~  c_Z \cdot A(X)$,
where $c_Z$ is a constant depending solely on $Z$. 
(In particular, when $Z$ is a lattice, $c_Z$ is the reciprocal of the area of the fundamental domain of $Z$.)
\end{claim}
\begin{proof}{\em (Sketch)~}
Assume w.l.o.g., that $X$ contains the origin, and let $\delta=Diam(X)$. 
Consider
$\Delta_n ~=~ \big|\,{{\sum_{z\in Z} A(\, Q_n\cap X +z)} - |Z\cap Q_n|\cdot A(X)\,\big|}$.
How big can it be? 
On the one hand,
\[
|Z \cap Q_{n-2\delta}|\cdot A(X)  ~\leq~     {\sum_{z\in Z} A(\, Q_n\cap X +z)} ~\leq~ 
 |Z \cap Q_{n+2\delta}|\cdot A(X)\,,
\]
and therefore 
\[\Delta_n ~\leq~ |Z \cap Q_{n+2\delta}|\cdot A(X) - |Z \cap Q_{n-2\delta}|\cdot A(X)
~=~ |Z \,\cap\, (Q_{n+2\delta}\setminus Q_{n-2\delta})| \cdot A(X)\,.
\]
On the other hand, since $Z \,\cap\, (Q_{n+2\delta}\setminus Q_{n-2\delta}) \,+\, X$ is contained in $Q_{n+4\delta}\setminus Q_{n-4\delta}$, covering no point there more than $d_{\cal F}$ times, it follows that $|Z \,\cap\, (Q_{n+2\delta}\setminus Q_{n-2\delta})| \cdot A(X)$ is at most $O(n)\cdot \delta \cdot d_{\cal F}$.
Hence, $\Delta_n = O(n)\cdot \delta \cdot d_{\cal F}$, and so $\Delta_n/n^2 \rightarrow 0$. 
The conclusion follows:
\[
\rho({\cal F}) ~=~ \liminf_{n\rightarrow \infty}\,  
{{\sum_{z\in Z} A(\, Q_n\cap X +z)} \over {n^2}} ~=~
\liminf_{n\rightarrow \infty}\,  
{{|Z \cap Q_n|\cdot A(X) \pm \Delta_n} \over {n^2}}  ~=~ 
\]
\[ \liminf_{n\rightarrow \infty}\,  
{{|Z \cap Q_n|} \over {n^2}} \cdot A(X)\, ~=~
c_Z \cdot A(X)\;.
\]
The fact that for a lattice $Z$, $\lim_{n\rightarrow \infty}\, {{|Z \cap Q_n|} / {n^2}}$ is the inverse of the 
the area of the fundamental domain of $Z$, is well known (see, e.g.,~\cite{rogers}).
 \end{proof}
The {\em covering density} of $X$, $\theta (X)\geq 1$, is defined as the 
infimum of $\rho({\cal F})$ over all covers of the form ${\cal F}=\{X + z\}_{z\in Z}$. 
It is well known (see, e.g.,~\cite{rogers}) that $\theta(X)$ is an affine invariant. 
\subsection{Odd Covers}
Let $X\subset \R^2$ be a compact set of a positive area $A(X)>0$. The family ${\cal F} = \{X + z\}_{z\in Z}$
for $Z \subseteq \R^2$ is called an {\em odd cover} of $\R^2$ if $X \circplus Z$ is well defined, and equals
to $\R^2$ up to a set of measure $0$. Notice that if $\F$ is an odd cover
of $\R^2$, then in particular it is a cover of $\R^2$. 
As before, we shall further require that the maximal degree of the cover of $\R^2$ by $\cal F$  is uniformly bounded.

The {\em odd covering density} of a compact $X$, $\theta^\circ (X)\geq 1$ is defined 
as the infimum of $\rho({\cal F})$ over all odd covers $\cal F$ as above. 
If no such $\cal F$ exists, set $\theta^\circ(X)=\infty$. Notice that $\theta^\circ(X) \geq \theta(X)$.
Similarly to the usual covering density $\theta (X)$, 
the odd covering density $\theta^\circ(X)$ is an affine invariant. This intuitively plausible 
statement can be proved formally along the same lines as the standard proof of the corresponding 
statement for the usual covers (see, e.g.,~\cite{rogers}). 
\footnote{
The requirement that $\cal F$ has a uniformly bounded 
degree does not appear in the standard definition of $\theta(X)$, despite the fact that
it is used in the proof of the affine invariance of $\theta(X)$ and elsewhere. The reason is that for 
any $\epsilon > 0$, a cover $\cal F$ can be easily modified into a {\em periodic} cover $\cal F'$ with $\rho({\cal F'}) \leq \rho({\cal F}) + \epsilon$, i.e., the corresponding $Z'$ is of the form $\Lambda + K$,
where $\Lambda$ is a lattice, and $K$ is finite (see, e.g.,~\cite{rogers}). Thus, w.l.o.g., one may restrict the discussion of $\theta(X)$ to periodic covers, 
and those are always uniformly bounded for a compact $X$. 
In contrast, the odd covers apparently do not allow such a modification, and so
the assumption about the uniformly bounded degree seems to be essential for them. 
This said, all odd covers occurring in this paper are periodic.
}
\subsection{Odd Compression Ratio: the Definition}
Let $X \subset \R^2$ be a compact set of area $0 < A(X) < \infty$.
Define $A_{odd}(X)$, the \emph{odd area} of $X$, 
to be the maximum number such that for any finite and odd-sized 
collection $\F$ of translates of $X$, the set of all points in $\R^2$ belonging to an odd 
number of members of $\F$ has area \;$\geq\, A_{odd}(X)$. I.e., $A_{odd}(X)$
is the infimum of $A(X \circplus K)$ over all finite odd-sized sets $K \subset \R^2$ (see \cite{P14, OPP15}). 

Define $\alpha^\circ(X)$, the \emph{odd compression ratio} of $X$,
as ${A_{odd}(X)} /{A(X)}$. Clearly, $0 \leq \alpha^\circ(X) \leq 1$, and it is an 
affine invariant.
%
\section{The Odd Cover Lemma}
\label{sec:odd-cov}
The following lemma, a variant, and in fact a special case, of 
Lemma\;1 from~\cite{OPP15}, is a useful tool 
for obtaining lower bounds on the odd compression ratio of $X$. 
For completeness, we provide two different
proofs for it. The first is shorter and simpler due to the preparation 
done in Section~{\ref{sec:dense}. It is a streamlined variant 
of the proof used in~\cite{OPP15}.
The second proof follows a somewhat different logic, and can be viewed as a generalization of 
the factor-space argument mentioned in the Introduction.
\begin{lemma}
\label{lemma:beta}
For any compact set $X$ of a positive measure in $\R^2$, the odd compression ratio of $X$
is at least the reciprocal of its odd covering density. That is,
\[
\alpha^\circ (X) ~\geq~ \theta^\circ(X)^{-1}~.
\]
\end{lemma} 
\begin{proof}{\it (A)}~~~ Let ${\cal F} = \{X + z\}_{z\in Z}$ be an odd cover of $\R^2$ of 
density $\rho({\cal F})$, and maximal cover degree $d_{\cal F} < \infty$. (If no such ${\cal F}$ 
exists, the lemma is trivially true.) Let $K \subset \R^2$ be any finite set  of odd cardinality.
Set $Y = X \circplus K$.

Consider the set $(X \circplus Z) \circplus K$.
On the one hand, it is equal to $\R^2$, up to a set of measure $0$. 
This is because $(X \circplus Z) = \R^2$, again up to a set of measure $0$,
and the cardinality of $K$ is odd.

On the other hand, using the commutativity of $\circplus$, one concludes that 
~$(X \circplus Z) \circplus K \,=\, (X \circplus K) \circplus Z \,=\, Y \circplus Z$. In other words, 
the family  ${\cal G} = \{Y + z\}_{z\in Z}$ is an odd cover of $\R^2$ of a maximal covering degree at 
most $d_{\cal F}\cdot |K|$. 

By Claim~\ref{eq:dense}, there is a constant $c_{Z}$ depending only on $Z$,
such that for every measurable set $W \subset \R^2$ such that $\{W + z\}_{z\in Z}$
is a cover of $\R^2$, it holds that  $\rho(\{W + z\}_{z\in Z})=c_Z \cdot A(W)$.
Therefore,
\[
1 ~\leq~ \rho({\cal G}) ~=~ c_Z \cdot A(Y) ~=~ \rho({\cal F}) \cdot {{A(Y)} \over {A(X)}}\,
~~~~~~~~\Longrightarrow~~~~~~~~ \rho({\cal F})^{-1}\;\leq\;  {{A(Y)} \over {A(X)}}~. 
\]
Taking the infimum over all odd-sized $K$'s to minimize $A(Y)/A(X)$, and the 
infimum over all legal $Z$'s to minimize $\rho({\cal F})$, one concludes that
$\theta^\circ(X)^{-1} ~\leq~ \alpha^\circ (X)$.
\end{proof}
\begin{proof}{\it (B)}~~~
Let ${\cal F} = \{X + z\}_{z\in Z}$ be an odd cover of $\R^2$ as before,
and let $S \subset \R^2$ be compact.  Consider the following mapping $\phi$ of 
the compact sets $S$ to the compact subsets of $X$:
\[
\phi(S) ~=~ \Sumodd_{z\in Z} \;  (S - z) \cap X ~=~ 
(S \circplus (-Z)) \cap X\,.
\]
\begin{claim}
\label{cl:phi}
$\mbox{}$
\begin{enumerate}
\item $~\phi(-X + a) ~=~ X$;  
\item $~\phi\left(\Sumodd_{i=1}^k S_i\right) =\; \Sumodd_{i=1}^k  \;\phi(S_i)$;  
\item $~A(\phi(S)) \leq A(S)\cdot \tilde{d}_{\cal F}(S)$,
where $\tilde{d}_{\cal F}(S)$ is the average degree of a cover of $S$ by $\cal F$,
i.e., the average of the cover degrees $d_{\cal F}(a)$, where $a$ ranges over $S$.
\end{enumerate}
\end{claim}
\begin{proof}  Indeed, for (1), keeping in mind that $X \circplus Z = \R^2$, and that 
$a-\R^2 = \R^2$, one gets
\[
\phi(-X + a) ~=~ \Sumodd_{z\in Z} \, (-X + a - z) \,\cap\, X 
                  ~=~ X \;\cap~  \Sumodd_{z\in Z} \,a+ (-X - z) 
                  ~=~ X \,\cap~~ a - (X\circplus Z) 
                  ~=~ X \cap\, \R^2 ~=~ X\,\!.
\]
For (2),
\[
\phi\left(\Sumodd_{i=1}^k S_i\right) = 
\Sumodd_{z\in Z} \left(\left(\Sumodd_{i=1}^k S_i ~- z \right) \cap X\right)
\;=\; \Sumodd_{z\in Z} \Sumodd_{i=1}^k (S_i - z) \cap X
\;=\; \Sumodd_{i=1}^k \Sumodd_{z\in Z} (S_i - z) \cap X \;=\; \Sumodd_{i=1}^k  \;\phi(S_i)\,.
\] 
For (3), observing that  $(S-z) \cap X = S\cap (X + z) - z$, one concludes that \vspace{3mm}
\\ 
$A(\phi(S)) 
~=~  A\left(\Sumodd_{z\in Z} [ S\cap (X+z) - z]\right)
~\leq~ \sum_{z\in Z} A( S\cap (X+z)) ~=~ A(S)\cdot \tilde{d}_{\cal F}(S)$.
\end{proof}
Instead of proving a lower bound on $\alpha^\circ(X)$, we shall prove one for 
$\alpha^\circ(-X+a)$, with a suitably chosen $a$. 
Since $\alpha^\circ(X)$ is invariant under affine transformations of $\R^2$,  
~$\alpha^\circ(-X+a) = \alpha^\circ(X)$. For typographical reasons, 
set $X_a = -X + a$. 

Consider, as before, any finite set $K \subset \R^2$ of odd cardinality, and let
$Y_a = X_a \circplus K$. On the one hand, by Claim~\ref{cl:phi}(3), 
~$A(\phi(Y_a)) \;\leq \; \tilde{d}_{\cal F}(Y_a) \cdot A(Y_a)$\,. On the other hand, by
Claim~\ref{cl:phi}(2)$\&$(1),
~$
\phi(Y_a) ~=~  \phi(\Sumodd_{k\in K} (X_a + k))  ~=~ \Sumodd_{k\in K} \phi(X_a + k) 
~=~ \Sumodd_{k\in K} X  ~=~ X\,.
$~
Thus, 
\[
\tilde{d}_{\cal F}(Y_a) \cdot A(Y_a)  ~\geq~ A(\phi(Y_a)) ~=~ A(X)
~~~~~~~~\Longrightarrow~~~~~~~~ 
{{A(Y_a)} \over {A(X_a)}} ~\geq~ \tilde{d}_{\cal F}(Y_a)^{-1}~.
\]
It remains to choose the translation vector $a$ as to minimize $\tilde{d}_{\cal F}(Y_a)$.
Getting back to the discussion of Section~\ref{sec:dense}, a simple averaging 
argument shows that for a random uniform $a\in Q_n$, the expected value of
$\tilde{d}_{\cal F}(Y_a)$ gets arbitrarily close to $\tilde{d}_{\cal F}(Q_n)$ as $n$ tends to infinity.
Keeping in mind the definition of $\rho_{\cal F}$, this implies in turn that there is a sequence of
$a$'s such that $\tilde{d}_{\cal F}(Y_a)$ approaches $\rho_{\cal F}$. 
Minimizing over all legal odd covers $\cal F$, one concludes that the infimum of 
$\tilde{d}_{\cal F}(Y_a)$ over $a\in \R^2$ is at most $\theta^\circ(X)$.
\end{proof}
To demonstrate the usefulness of Lemma~\ref{lemma:beta}, assume that there is a tiling
of $\R^2$ by translates of $X$. Then, $\theta^\circ(X)=1$, implying $\alpha^\circ(X) = 1$.
This yields the aforementioned result about the non-compressibility of the square
and the centrally symmetric hexagon.

Further, assume that $X$ is a triangle $(a,b,c)$.
Let $\Lambda$ be the lattice spanned by $\{ \half (b-a),\, \half (c-a) \}$. Then,
${\cal F} = \{X + z\}_{z\in \Lambda}$ is an odd cover of $\R^2$ covering each point in the plane
either 1 or 3 times, with $\rho({\cal F})=2$. This implies $\alpha^\circ(X) \geq \half$,
matching the optimal bound of~\cite{P14}.
%
\section{Odd Covers by Stripe Patterns}
A \emph{stripe pattern} is a (non-singular) affine image of the set 
$\{(x,y) \in \mathbb{R}^2 \mid \mbox{$\lfloor y \rfloor$ is even}\}$.
I.e., it is an infinite set of parallel stripes of equal width $w$, such that the 
distance between any two adjacent stripes is $w$ as well
(see Figure \ref{figure:stripes}). The {\em direction} of a stripe 
pattern is, expectedly, the direction of a boundary line of any stripe in it.
The \emph{width} of the stripe pattern is the $w$ as above. 

\begin{figure}[ht]
    \centering
    \includegraphics[width=7cm]{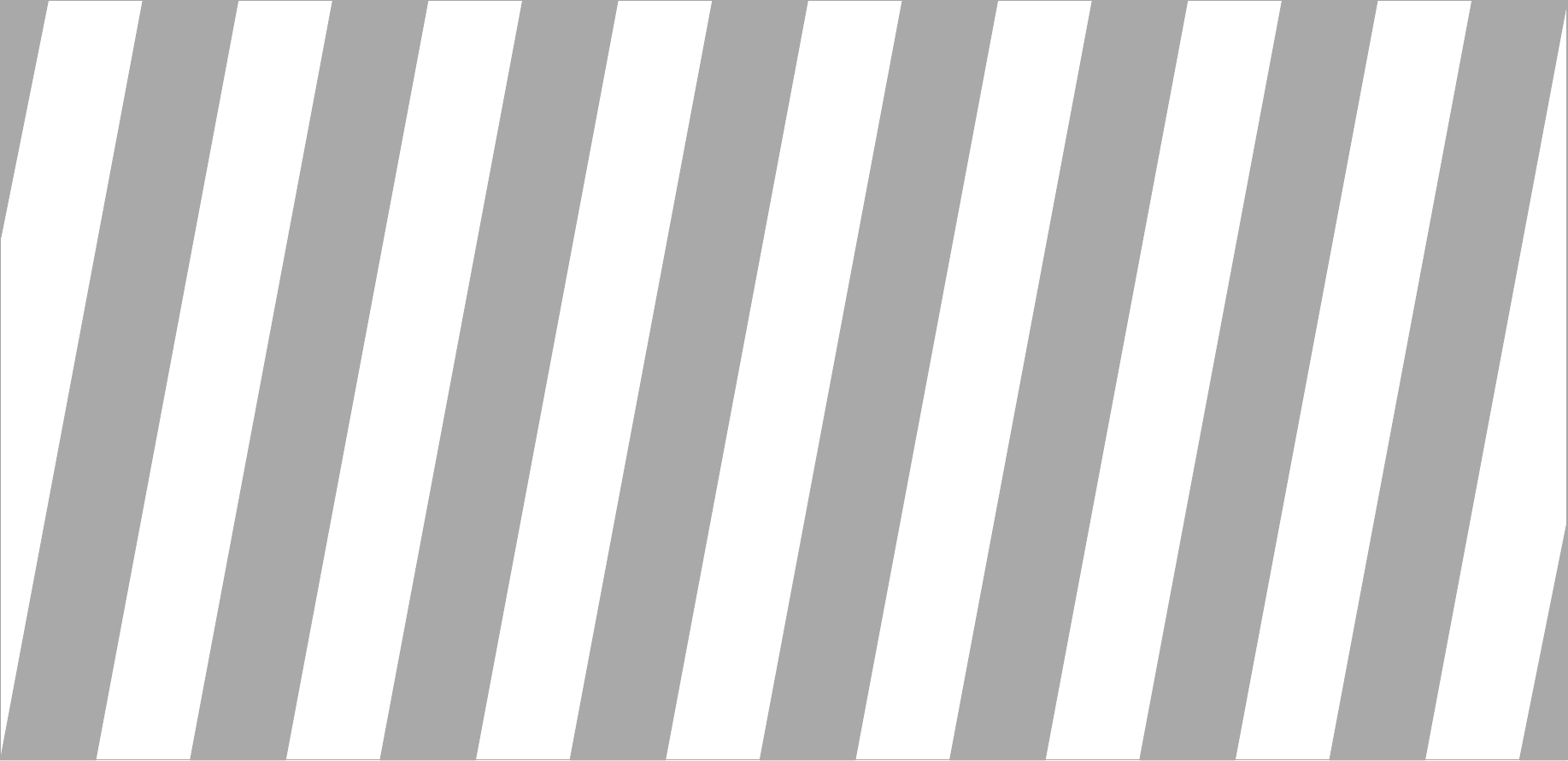}
        \caption{A stripes pattern}
        \label{figure:stripes}
\end{figure}

We start with the following simple but useful observation about stripe patterns.
The easy verification is left to the reader.

\begin{observation}\label{observation:stripes}
Let $S$ be a stripe pattern, and let $\ell$ and $\ar$ be the two lines
delimiting one of the stripes in $S$. Then, for every $a \in \ell$, $b \in \ar$, 
and $v=b-a$, it holds that:
\begin{enumerate}
\item $S \circplus \{0,v\} ~=~ S \sumodd (S+v) ~=~ \mathbb{R}^2$.
%
\item $S \circplus \{0,\half v\} = S \sumodd (S+\frac{1}{2}v)$ is a stripe pattern 
with the same direction as $S$, whose width is equal to half of the width 
of $S$.  
\end{enumerate}
\end{observation}
The main result of this section is:
\begin{lemma}\label{lemma:stripes}
Let $S_{1}, \ldots, S_{k}$ be stripe patterns with pairwise distinct directions,
and let $T=S_{1} \sumodd \cdots \sumodd S_{k}$.
Then, there exists a finite (and efficiently computable) set of vectors 
$U\subset \R^2$, $|U| \leq 2^{k-1}$,  such that 
$\,T \circplus U = \Sumodd_{u_i\in U}\,(T+u_{i}) = \R^2$, up to a set of measure $0$. 
\end{lemma}
It will be technically more convenient to prove the following 
more general statement:
\begin{lemma}\label{lemma:stripes2}
Let $S_{1}, \ldots, S_{k}$ be stripe patterns with pairwise distinct directions,
and let $\{Z_i \}_{i=1}^k$ be a family of finite nonempty subsets of $\mathbb{R}^2$, 
with $Z_1 = \{0\}$.
Let $T=\Sumodd_{i=1}^{k}(S_{i} \circplus Z_i)$.~
Then, as before, there exists a finite (and efficiently computable) set of vectors 
$U \subset \R^2$,  $|U| \leq 2^{k-1}$, such that 
$\,T \circplus U = \Sumodd_{u_i \in U}\,(T+u_{i}) = \R^2$, up to a set of measure $0$. 
\end{lemma}
Lemma~\ref{lemma:stripes} follows from Lemma~\ref{lemma:stripes2}
by setting $Z_i =\{0\}$ for all $i\geq 2$.

Notice the special role of $S_{1}$ in the statement of Lemma~\ref{lemma:stripes2}.
In fact, the condition $Z_{1}=\{0\}$ is essential even for $k=1$. It is easy to verify that, using the notation of Observation~\ref{observation:stripes}, no finite set of translates of the set 
$T = S_1 \circplus \{0,\frac{2}{3}v\}$ can oddly cover the plane\footnote{Perhaps expectedly,  the same $T$ has also the complementary extremal property: $T \circplus \left\{0, \frac{1}{3}v, \frac{2}{3}v \right\} = \emptyset$.}.

\begin{proof} {\em (of Lemma \ref{lemma:stripes2})~~}
For every $i=1,2,\ldots, k$, let $\ell_i$ and $\ar_i$ denote the two parallel 
lines delimiting some stripe in $S_{i}$. By the assumptions of the Lemma, 
for different $i$'s these have different directions, and therefore
intersect.

The proof proceeds by induction on $k$.
 
For $k=1$, the statement follows from Observation~\ref{observation:stripes}(1). 

For $k=2$, let $a$ and $b$ be the intersection points of $\ell_1$ and $\ar_1$ 
with $\ell_2$, respectively. Setting $v_2=b-a$, we have 
$S_{1} \circplus \{0,v_2\}=\R^2$, by Observation~\ref{observation:stripes}(1).
Moreover, since $v_{2}$ has the same direction as of $S_{2}$, we have
$S_{2} \circplus \{0,v_2\}=\emptyset$. Keeping this in mind we have:
\[
T \circplus \{0,v_2\} ~=~ \left(S_1 \sumodd (S_2 \circplus Z_2)\right) \circplus \{0,v_2\}
~=~ \left(S_1 \circplus \{0,v_2\} \right) \sumodd 
\left(S_2 \circplus Z_2 \circplus \{0,v_2\}\right) ~= 
\]
\[
=~ \R^2 \sumodd \left(S_2 \circplus \{0,v_2\}\circplus Z_2 \right) ~=~ 
\R^2 \sumodd (\emptyset \circplus Z_2) ~=~ \R^2 \sumodd \emptyset ~=~ \R^2\,.
\]

For $k > 2$, we proceed as follows. Let $v_k$ be the (well-defined) vector such that, 
on the one hand, $\ell_k + v_k = \ar_k$, and on the other hand, $\ell_1 + 2v_k = \ar_1$.
Observation~\ref{observation:stripes}(1) implies that  
$S_{k} \circplus \{0,v_k\} = \R^2$, and hence 
$(S_{k} \circplus Z_k) \circplus \{0,v_k\}$ equals $\R^2 \circplus Z_k$, which is 
either $\emptyset$ or $\R^2$, depending on the parity of $Z_k$.
Observation~\ref{observation:stripes}(2) implies that $S_{1} \circplus \{0, v_k\}$ is a 
stripe pattern with the same direction as $S_{1}$, and half its width. 
Consequently,  
~$
\left(S_1 \circplus \{0,v_k\}\right) \sumodd \left( (S_{k} \circplus Z_k) \circplus \{0,v_k\}\right)
$\,
is a stripe pattern with the same direction as $S_{1}$ and half its width as well.

Consider now the set $T'=T \circplus \{0,v_k\} = T \sumodd (T+v_k)$. 
Using the properties of the operators $\sumodd$ and $\circplus$, one gets:
\begin{eqnarray}
\label{eq:XOR}
T'&=& T \circplus \{0,v_k\}  ~= ~ \left(\Sumodd_{i=1}^{k} (S_{i} \circplus Z_i) \right) 
\circplus \{0,v_k\} ~=~ \Sumodd_{i=1}^{k} (\,S_{i} \circplus Z_i \circplus \{0,v_k\}\,) 
\end{eqnarray}
As we have just seen, the $\sumodd$ of the first and the $k$'th terms of the latter sum 
is a stripe pattern $S'_1$ with the same direction as $S_{1}$. Thus, setting 
$Z_i^{'} = Z_i \circplus  \{0,v_k\}$, one arrives at
\begin{equation}\label{eq:T'2}
T' ~=~ S'_{1} ~\sumodd~ \Sumodd_{i=2}^{k-1} (S_{i} \circplus Z'_i)
\end{equation}

By the induction hypothesis applied to $T'$, there exists a finite
set $U' \subset \R^2$ such that $T' \circplus U' = \R^2$ up to a set of measure $0$. 
However,
\begin{equation}\label{eq:T'3}
T' \circplus U'  ~=~ T  \circplus \{0,v_k\} \circplus U' ~=~ T \circplus (U' \circplus \{0,v_k\}) 
\end{equation}
Therefore, setting $U= U' \,\circplus\, \{0,v_k\}$, one concludes that 
\,$T \circplus U =  T' \circplus U' = \R^2$. This completes the construction of the desired set $U$.

It remains to estimate the size of $U$. The recursive definition~ $U = U' \,\circplus\, \{0,v_k\}$ for $k> 2$, combined with the base cases $|U|=2^{k-1}$ for $k=1,2$, implies
the desired bound: $|U| \leq 2^{k-1}$. 
\end{proof}
%
\section{Odd Covers by Rational Polygons: A Special Case}
In this section we prove our main theorem for the special case of  
rational polygons with no two parallel edges.

Given a rational polygon $\P$, let $\P_{\rm INT}$ be the integer polygon with
minimal area affinely equivalent to $\P$, and let $A_{\rm INT}(\P) = A(\P_{\rm INT})$
be its area.
\begin{theorem}\label{theorem:pnp}
Let $\P$ be a rational polygon with $k$ vertices, and no parallel edges.
Then, there exists a bounded degree odd cover $\cal F$ of $\R^2$ by translates of $\P$
with density $\rho({\cal F}) \leq A_{\rm INT}(\P)\cdot 2^{k-1}$.
Consequently,~
$\alpha^\circ (\P) \geq A_{\rm INT}(\P)^{-1}\cdot 2^{-(k-1)}$.
\end{theorem}
Before starting with the proof, we need one more observation about the structure of
$\sumodd$-sums of stripe patterns.
For $i=1,\ldots\,r$, let $L_{i}$ be an affine image 
of the family of parallel lines~$\{ (x,y) \in \R^2 \,|\, y \in\mathbb{Z}\}$. 
Respectively, let $S_i$ be a stripe pattern whose boundary is $L_i$.
(Notice that there are exactly two such stripe patterns: $S_i$ and its complement
$\overline{S}_{i} = \R^2 \setminus S_i$.)
Assume that $S_1, \ldots, S_{r}$ have pairwise distinct directions. 
The union of all these lines $\bigcup_{i=1}^{r}L_i$ partitions $\R^2$ into pairwise disjoint open 
cells, each cell being a convex polygon. Call two cells {\em adjacent} if they 
share a $1$-dimensional edge. 

It is a folklore to show that the cells of $\R^2 \setminus \bigcup_{i=1}^{r} L_i$
can be $2$-colored in such a way that any two adjacent cells have different colors.

\begin{claim}\label{claim:stripes2}
Let $T$ be the union of all cells of $\R^2 \setminus \bigcup_{i=1}^{r} L_i$ in one
color class. Then, (up to the $0$-measure boundary of $T$, i.e., $\bigcup_{i=1}^r L_i$)
either $T=S_{1} \sumodd \cdots \sumodd S_{r}$, or 
$T= \R^2 \setminus (S_{1} \sumodd \cdots \sumodd S_{r}) =
\overline{S}_{1} \sumodd S_2 \sumodd\cdots \sumodd S_{r}$.
\end{claim}

The claim is rather obvious, and can be formally verified, e.g., 
by induction on $r$. The full details are left to the reader (see Figure~\ref{figure:sss} for an illustration).

\begin{figure}[ht]
    \centering
    \includegraphics[width=11cm]{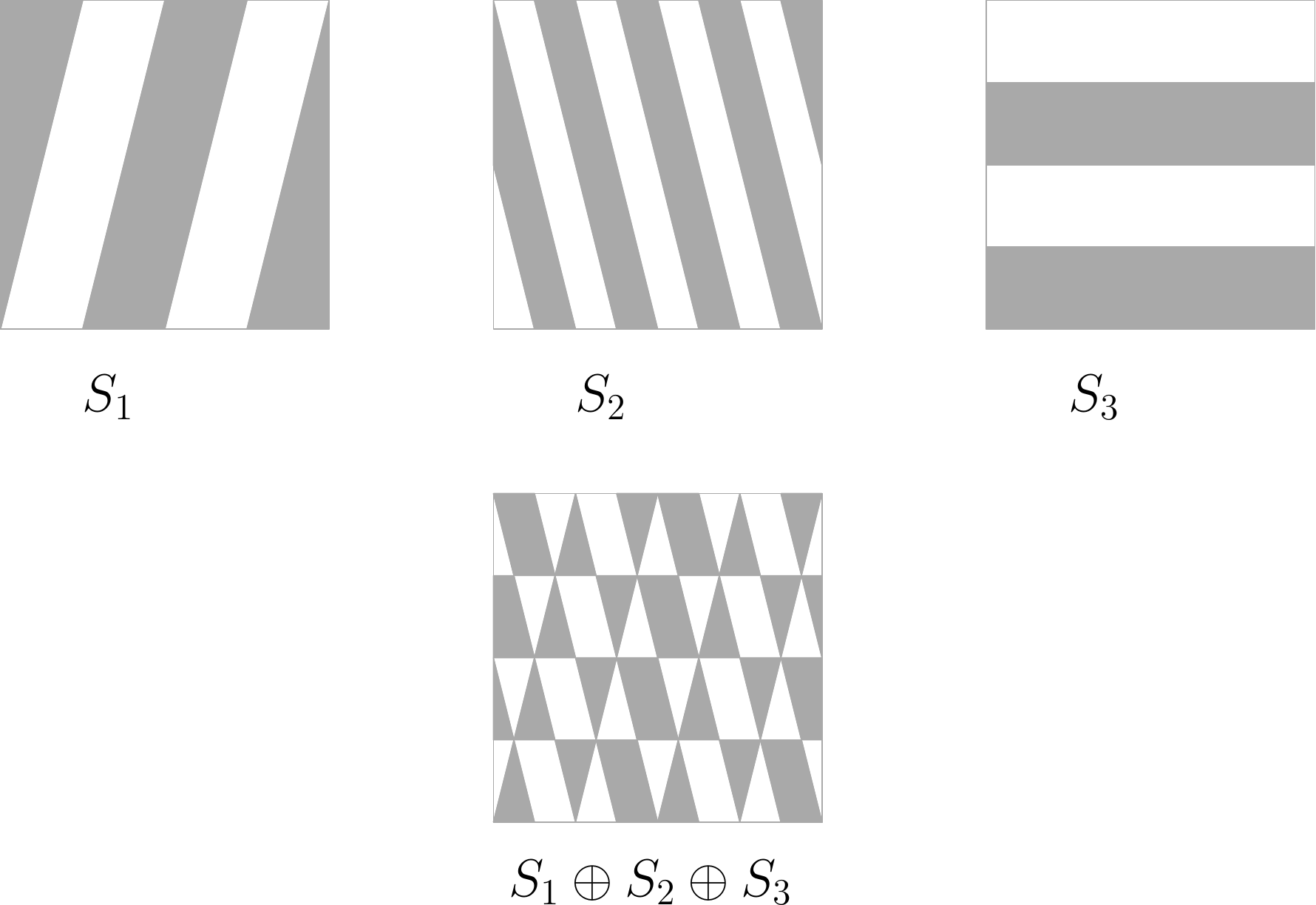}
        \caption{$\oplus$-sum of three stripes patterns}
        \label{figure:sss}
\end{figure}


\begin{proof} {\em (of Theorem~\ref{theorem:pnp})}~~
Keeping in mind that both $\theta^\circ(\P)$ and $\alpha^\circ(\P)$ are affine invariants, one may assume without loss of generality that 
$\P = \P_{\rm INT}$, and that the origin $O=(0,0)$ is a vertex of $\P$. 
Then, all the vertices of $\P$ belong to $\Z^2$. Observe also that some of the edges 
of $\P$ must contain an even number of integer lattice points. 
Otherwise, the coordinates of the vertices of $\P$ would all have 
the same parity, i.e., they would all be even.
Scaling such an all-even $\P$ by a factor of $\half$ would have yielded a smaller integer polygon affinely 
equivalent to $\P$, contrary to the definition of $\P_{\rm INT}$. 

We claim that $\P \circplus \Z^2$ is equal to 
$S_{1} \sumodd \ldots \sumodd S_{r}$, where $S_{1}, \ldots, S_{1}$ are stripe
patterns with pairwise distinct directions, and $r$ is at most the number of 
vertices ($=$edges) of $\P$. Once this claim is established, the rest easily follows.

Indeed, assuming that the claim holds, by Lemma~\ref{lemma:stripes} there exists $U \subset \R^2$ 
with $|U| \leq 2^{r-1}$ such that~ $(\P \circplus \Z^2) \circplus U \;=\; \R^2$. 
Equivalently, the (multi-) family of sets
~$
\F ~=~ \{\P+z+u\}_{\,z\in \Z^2,\,u\in U}
$~
is an odd cover of the plane. To employ the Odd Cover Lemma~\ref{lemma:beta}, one needs to
estimate the density of this cover.
Observe that $\{\P+z\}_{\,z\in \Z^2}\,$ has a bounded maximal degree (being the maximal 
number of integer lattice points in any translate of $\P$), while its average density is $A(\P)$, 
as mentioned in Claim~\ref{eq:dense}. Therefore, the maximal degree of $\F$ is at most $|U|$ times 
the maximal degree of the cover $\{\P+z\}_{\,z\in \Z^2}$, while 
$\rho(\F)$, the average degree of $\F$, is precisely $A(\P)\cdot |U| \leq A(\P) \cdot 2^{k-1}$. Hence,
~$
\theta^\circ(\P) ~\leq~ \rho(\F) ~\leq~ A(\P)\cdot 2^{k-1}\;.
$
Applying the Odd Cover Lemma~\ref{lemma:beta} one gets 
~$\alpha^\circ(\P) ~\geq~ \theta^\circ(\P)^{-1} ~\geq~ A(\P)^{-1}\cdot 2^{-(k-1)}$,~ as needed.

$\mbox{}$
\\
Thus, it is sufficient to show that $\P \circplus \Z^2$ is equal to 
$S_{1} \sumodd \ldots \sumodd S_{r}$ as above. In the remainder of this section,
we shall focus on proving this claim. The {\bf argument} goes as follows.

Let $E(\P)$ denote the set of all edges of $\P$. For $e\in E(\P)$, let $L_{e}$ be the set of all lines 
parallel to $e$ that contain points of $\Z^2$. Clearly, $L_e$ is a discrete set of lines as in
Claim~\ref{claim:stripes2}.  
Consider a point $x \in \R^2$. It belongs to 
$\P \circplus \Z^2$ exactly when $|{\P} \cap (x-\Z^2)|$ is odd.
Unless $x \in \bigcup_{{\footnotesize{e \in E(\P)}}} L_{e}$,
every point $x'$ in a sufficiently small neighborhood of $x$ will 
satisfy $|{\P} \cap (x-\Z^2)|=|{\P} \cap (x'-\Z^2)|$.
Therefore, $\,\P \circplus \Z^2\,$ is a union of cells of 
$\R^2 \setminus \bigcup_{{\footnotesize{e\in E(\P)}}} L_e$. 

Call an edge $e$ of $\P$ {\em active} if it contains an even number of integer 
lattice points, and {\em passive} otherwise. 
Respectively, if $e$ is active, all the lines in 
$L_{e}$ are called {active}, and if it is passive, 
the lines in $L_{e}$ are called {passive}.

Let $C_{1}$ and $C_{2}$ be two adjacent cells in 
$\R^2 \setminus \bigcup_{{\footnotesize{e\in E(\P)}}} L_e$ separated by
a line $\ell \in L_{e}$ for some edge $e$ of $\P$.
We claim that if $e$ is active, then exactly one of $C_{1}$ and $C_{2}$
is contained in $\P \circplus \Z^2$, and if $e$ is passive, then either both 
are contained in $\P \circplus \Z^2$, or none of them is.

Indeed, let $I \subset \ell$ denote the common $1$-dimensional edge
of $C_{1}$ and $C_{2}$. Observe that the only members in the family 
$\F=\{P + z\}_{z \in \Z}$ that distinguish between $C_{1}$ and $C_{2}$, that is,
contain exactly one of the two, are those that contain $I$ in their boundary. 
To get a clearer picture of this subfamily, let $J=[p,q] \subset \ell$ be the smallest interval with 
integer endpoints containing $I$. Notice that $I$ has no integer points in its interior.
Let us view $e$ as a $1$-dimensional interval $[s_e,t_e) \subset \R^2$,
parallel to, and having the same orientation as, $J$ . Then, $P+z$ contains $I$
if and only if $p \in e+z$. Or, equivalently, $p-z\in e$.

This means that when $e$ is active (i.e., it contains an even number of points in $\Z^2$),
$I$ is covered by an odd number of $(P+z)$'s, and when $e$ is passive, it is covered by 
an even number of $(P+z)$'s. Consequently, in the former case the  
degrees of cover of the cells $C_1$ and $C_2$ by $\F$ have a different parity,
whereas in the latter case the parities are equal. Thus, when $e$ is active, $\P \circplus \Z^2$
distinguishes between $C_1$ and $C_2$, and when it is passive, it does not. As claimed.

Let $AE(\P)$ be the (nonempty!) set of active edges of $\P$. 
The conclusion is that $\P \circplus \Z^2$ is a union of cells of 
$\R^2 \setminus \bigcup_{{\footnotesize{e\in AE(\P)}}} L_e$, satisfying
the assumptions of Claim~\ref{claim:stripes2}. Hence, $\P \circplus \Z^2$ is a $\sumodd$-sum 
of stripe patterns, as desired. This completes the proof of Theorem~\ref{theorem:pnp}.
\end{proof}

The assumption that $\P$ has no parallel edges was needed to justify the (tacit)
assumption that for every line $\ell \in \bigcup_{\footnotesize{e \in E(\P)}} L_e$, there is 
a {\em unique} edge $e$ such that any translate of $\P$ may have contained 
in $\ell$.
When there are parallel edges, most of the argument still applies, however, it may
fail at one fine point. The contributions of parallel edges may cancel out, 
leaving no active lines, and resulting in $\P \circplus \Z^2 = \emptyset$. 
Unfortunately, this 
situation indeed does occur for some rational polygons $\P$, for example, for the centrally 
symmetric ones.
To overcome this problem, a more refined family of translates will be constructed. 
%
%
%
%
%
\section{A Theorem About $\Z_2$-valued Functions on Integer Lattices}  
We shall need the following result of an independent interest. It will be proven here 
for any dimension $d$, but used in Section~\ref{sec:main} only with $d=2$.

Let $\A$ be a family of finite subsets of $\Z^d$. A function, or, rather, a weighting, 
$\Fgoth:\Z^d \rightarrow \Z_2$, will be called {\em stable}\, with respect to $\A$, 
if for any $A \in \A$, all integer translates of $A$ have the same $\,\Fgoth$-weight. 
That is, the value of $\,\Fgoth(A + p) = \Sumodd_{x\in A+p} \Fgoth(x)$, 
does not depend on the choice of $p\in \Z^d$, but solely on $A$.\footnote{
In this section, the operator $\sumodd$ that was originally defined on sets, will be
sometimes applied to points. For consistency, regard points as single-element sets.}
Further, call $\Fgoth$ {\em \,$0$-stable}
with respect to $\A$, if it is stable, and moreover, for every  $A \in \A$, $\,\Fgoth(A) = 0$. For example, if the function $\Fgoth$ is everywhere $0$, then it is $0$-stable with 
respect to any family $\A$.  If it is everywhere $1$, it is stable with respect to any family $\A$,
and $0$-stable if $\A$ consists only of sets of even cardinality.
\begin{theorem}\label{theorem:f}
Let $\A$ be a (possibly infinite) 
family of non-empty finite subsets of $\Z^d$, and $\A \neq \emptyset$. 
There exists a function $\Fgoth:\Z^d \rightarrow \Z_2$ that is
stable, but not $0$-stable, with respect to $\A$.
\end{theorem}
\begin{proof}
We start with the $1$-dimensional case, introducing the key construction
to be used in all dimensions. 
\\ \\
{\bf\underline{Case $d=1$}}\\
We define a family $\{{\fgoth}_k\}_{k=0}^\infty$ of functions from $\Z$ to $\Z_2$
in the following recursive manner. We will show that one of this functions
is the desired function $\Fgoth$:
\\$ $ \par
{\bf
$\fgoth_0$ is identically $1$;

For $k>0$, ~~$\fgoth_{k}(0) = 1$, and 
$\fgoth_{k}(t) \,=\, \fgoth_{k}(t-1) \sumodd \fgoth_{k-1}(t-1)$. 
\footnote{
Observe that this recursive formula defines $\fgoth_{k}(t)$ for both positive and negative values 
of $t$. More explicitly, for $t<0$ it becomes $\fgoth_{k}(t) \,=\, \fgoth_{k}(t+1) \sumodd 
\fgoth_{k-1}(t)$, reducing either $k$ or $|t|$ just as for $t>0$. 
}
}
\\ \\
For example, $\fgoth_{\small 1}(t)$ is $1$ if $t$ is even, and $0$ otherwise. The next one, $\fgoth_{\small 2}(t)$,
is $1$ if $t\equiv 0,3~({\rm mod~ 4})\,$, and $0$ otherwise. 
Observe that the repeated application of the recursive formula yields for any $c\in \N$,
\begin{equation}
\label{eq:f}
\fgoth_{k}(t+c) ~=~ \fgoth_{k}(t) \; \oplus\; \Sumodd_{i=0}^{c-1} \fgoth_{k-1}(t+i)\,.
\end{equation}
\begin{claim}\label{claim:f1}
If $\fgoth_{k-1}$ is $0$-stable with respect to a finite $A\subseteq \Z$, 
then $\fgoth_{k}$ is stable with respect to $A$.
\end{claim}
Indeed, it suffices to show that for any $p\in \Z$, 
$\fgoth_{k}(A + p +1) = \fgoth_{k}(A + p)$. By definition of $\fgoth_{k}$,
\[
\fgoth_{k}(A + p +1) ~= \Sumodd_{t\in A+p+1} \fgoth_{k}(t) ~=
\Sumodd_{t\in A+p+1} \fgoth_{k}(t-1) \;\sumodd \Sumodd_{t\in A+p+1} \fgoth_{k-1}(t-1)
~=~ \fgoth_{k}(A+p) \;\sumodd \;\fgoth_{k-1}(A+p)~.
\]
Since $\fgoth_{k-1}(A+p)=0$ by assumptions of the claim, one concludes that 
$\fgoth_{k}(A + p +1) = \fgoth_{k}(A + p)$.
\begin{claim}\label{claim:f2}
For any $k\geq 1$, ~$\fgoth_{k}(t)=0$\; for\; $1\leq t\leq k$.
\end{claim}
Indeed, apply induction on $k$.
For $k=1$, ~$\fgoth_{1}(1) = \fgoth_{1}(0) \sumodd \fgoth_{0}(0) = 1\sumodd 1=0$.
For $k>1$, using (\ref{eq:f}), one concludes that for any $t$ in the range,
\[
\fgoth_{k}(t) ~=~ \fgoth_{k}(0) \sumodd \fgoth_{k-1}(0) \sumodd \fgoth_{k-1}(1) \sumodd \ldots \sumodd \fgoth_{k-1}(t-1) ~=~ 1\sumodd 1\sumodd 0 \sumodd \ldots \sumodd 0 ~=~ 0.
\]

We proceed to show that one of $\fgoth_k$'s satisfies the requirements of the theorem.
Observe that $\fgoth_0$ is stable with respect to $\A$. By Claim~\ref{claim:f1}, 
either there exists $k\geq 0$ such that $\fgoth_k$ is stable, but not $0$-stable (precisely as desired), or all $\fgoth_k$'s are $0$-stable. However, the latter situation does not
occur. Consider any nonempty $A\in \A$, and let 
$a=\min(A)$, $b=\max(A)$, $r=b-a$. 
Then, by Claim~\ref{claim:f2}, ~$\fgoth_r(A-a)=1$, and thus $\fgoth_r$
is not $0$-stable. This completes the case $d=1$.
\\ \\
{\bf\underline{General Case}}\\
Let $L$ be a linear function (without a constant term) from $\Z^d$ to $\Z$ that satisfies two requirements.
The first requirement is that the coefficient of $x_1$ in $L$ is 1. The second requirement
is that for some $A\in \A$, ~$L$ attains a minimum on $A$ at a unique point.
Such $L$'s exist. E.g., assuming that $A$ can be translated to a subset of a 
of the cube $[0,r-1]^d$, the function $\sum_{i=1}^d r^{i-1}x_i$ is one-to-one on $A$ by
the uniqueness of the base-$r$ representation, and so its minimum on $A$ is attained exactly once.   

For $k\geq 0$ and $a\in\Z^d$, define $\Fgoth_k(a) \,=\, \fgoth_k(L(a))$.
Respectively, for a finite subset $A \subset \Z^d$, define 
$\Fgoth_k(A)\,=\,\Sumodd_{a\in A} \Fgoth_k(a)\,$.

The proof proceeds along the same lines as in the $1$-dimensional case.
\begin{claim}\label{claim:F1}
If \;$\Fgoth_{k-1}$ is $0$-stable with respect to a finite $A\subseteq \Z^d$, 
then $\Fgoth_{k}$ is stable with respect to it.
\end{claim}
It suffices to show that for any $p\in \Z^d$, and any unit vector $e\in \Z^d$,
$\Fgoth_{k}(A + p +e) = \Fgoth_{k}(A + p)$. Let $L(e)=c$. If $c=0$, the statement 
is trivial. If $c<0$ the statement reduces to the case $c>0$ by considering $-e$ instead
of $e$. Thus, without loss of generality, $c>0$. By~(\ref{eq:f}), the linearity of $L$,
and the first requirement on it, 
\[
\Fgoth_{k}(A + p +e) ~= \Sumodd_{t\in A+p+e} \fgoth_{k}(L(t)) ~=
\Sumodd_{t\in A+p} \fgoth_{k}(L(t) +c) ~=
\Sumodd_{t\in A+p} \fgoth_{k}(L(t)) ~~\sumodd~~ 
\Sumodd_{i=0}^{c-1}\Sumodd_{t\in A+p} \fgoth_{k-1}(L(t)+i) ~=
\]
\[ 
=~~ \Fgoth_k(A+p) ~~\sumodd~~ \Sumodd_{i=0}^{c-1}\Sumodd_{t\in A+p +i\cdot e_1} \fgoth_{k-1}(L(t)) ~=~ \Fgoth_k(A+p) ~~\sumodd~~  
\Sumodd_{i=0}^{c-1}  \Fgoth_{k-1}(A+p +i\cdot e_1)\,.
\]
Since $\Fgoth_{k-1}$ is $0$-stable with respect to $A$, the second summand is $0$, and thus $\Fgoth_{k}(A + p +e) = \Fgoth_{k}(A + p)$. This concludes the proof of Claim~\ref{claim:F1}.

To conclude the proof of the theorem, observe that $\Fgoth_0$ is stable with respect to $\A$,
and thus, by Claim~\ref{claim:F1}, either there exists $k\geq 0$ such that $\Fgoth_k$ is stable, 
but not $0$-stable, precisely as desired, or all $\Fgoth_k$'s are $0$-stable. 

As before, the second possibility does not occur. Indeed, by the second requirement
on $L$, there exists $A\in \A$ on which $L$ attains a unique minimum.
Let  $p\in A$ be the point on which the minimum is attained, and let $a\in\Z$ denote its value.
Then, $L(A -a\cdot e_1 )= L(A)-a$~ is a subset of $[0,k]$ for some $k\in\N$, 
and $0$ has a unique pre-image $p'=p -a\cdot e_1$. By Claim~\ref{claim:f2}, 
~~$\Fgoth_k(A-a\cdot e_1) ~=~ \fgoth_k(0) ~\sumodd~  
\Sumodd_{t\in A \setminus p'}\;\fgoth_k(L(t)) ~=~ 1\sumodd 0 ~=~ 1$.
\end{proof}
\section{The Main Theorem}
\label{sec:main}
We can now prove the main theorem in full generality, making no assumption about
parallel edges.
\begin{theorem}\label{theorem:main}
Let $\P$ be a rational polygon with $k$ distinct classes of parallel edges.
Then, there exists a bounded degree odd cover $\cal F$ of $\R^2$ by translates of $\P$
with density $\rho({\cal F}) \leq A_{\rm INT}(\P)\cdot 2^{k-1}$.
Consequently,~
$\alpha^\circ (\P) \geq A_{\rm INT}(\P)^{-1}\cdot 2^{-(k-1)}$.
\end{theorem}
\begin{proof}
While the family of translates will in general be different from the one
used in the proof of Theorem~\ref{theorem:pnp}, the logical structure of the 
proof will be essentially identical.
Let us re-examine this structure.

Assuming that $\P=\P_{INT}$, the first and main goal is to construct a family 
$\F=\{\P + z\}_{z\in Z}$\,, $Z \subseteq \Z^2$, such that 
${\P} \circplus Z$ is a $\sumodd$-sum of at most $k$ stripe patterns. 
(In the former proof, $Z$ was just $\Z^2$.) 
A close reading of the proof of Theorem~\ref{theorem:pnp} 
reveals that in order to prove this fact about $\F$, it is sufficient to show 
that $\F$ has a certain property. To formulate it we need some definitions.

Let $D(\P)$ be the set of all the classes of parallel edges of $\P$, or simply 
the {\em directions} of $\P$. For every $d \in D(\P)$, let $L_d$ be the set of all lines 
in $\R^2$ in the direction of $d$ that contain integer lattice points.
As before, each $L_{d}$ is a discrete set of parallel 
lines with equal distances between any two consecutive ones.

Consider the arrangement of lines $\bigcup_{d} L_{d}$. Observe that since
$Z \subseteq \Z^2$, for any $(\P + z) \in \F$, and any open cell $C$ of 
$\R^2 \setminus \bigcup_{d} L_{d}$\,, either $C \subseteq (\P + z)$, or $C \cap (\P + z) = \emptyset$.
\begin{property}
\label{pr:crucial}
Let $P,D(P),Z,{\cal F},\bigcup_{d} L_{d}$ be as above. 
Further, let $I$ denote an edge of the arrangement $\bigcup_{d} L_{d}$ (i.e., a common 1-dimensional
boundary of two adjacent cells) that lies on a line $\ell \in L_{d}$, $d \in D(\P)$. 
The family $\cal F$ has the desired property if:
\begin{enumerate}
\item
For {\em any} edge $I$ of the arrangement $\bigcup_{d} L_{d}$ as above,
the parity of the number of sets in $\F$ 
whose boundary contains $I$ depends solely on the corresponding direction $d$. 
\item Moreover, there exists $d\in D(\P)$ such that this parity is odd. We call such a direction 
{\em active}, and denote the set of all active directions by $AD(\P)$.
\end{enumerate}
\end{property}

Once Property~\ref{pr:crucial} is established for $\F =\{\P + z\}_{z\in Z}$,
the {\bf argument} from the proof of Theorem~\ref{theorem:pnp} 
implies that $\P \circplus Z$ is a (nonempty) union of cells of 
$\R^2 \setminus \bigcup_{\footnotesize{d \in AD(\P)}} \,L_d$
that satisfy the assumptions of Claim~\ref{claim:stripes2}.
Applying Claim~\ref{claim:stripes2}, one concludes that $\P \circplus Z$ is equal to 
$S_{1} \sumodd \ldots \sumodd S_{r}$ for some stripe patterns 
$S_{1}, \ldots, S_{r}$, and $r=|AD(\P)|$, the number of active directions, 
is at most $|D(\P)| = k$.

Once the main goal is achieved, the rest is easy.
Lemma~\ref{lemma:stripes} is used to conclude that
there exists a finite set $U \subset \R^2$ with $|U| \leq 2^{r-1}$, 
such that~ $(\P \circplus Z) \circplus U \;=\; \R^2$. Equivalently, the 
(multi-) family of sets
~$
\F ~=~ \{\P+z+u\}_{\,z\in Z,\,u\in U}
$~
is an odd cover of the plane. Since $Z$ is a subset of $\Z^2$,  
the density of this odd cover is at most 
$A_{\rm INT}(\P)\cdot |U| \leq A_{\rm INT}(\P)\cdot 2^{k-1}$. (For the appearance
of $A_{\rm INT}$, consult Claim~\ref{eq:dense}.) Finally, by the Odd Cover 
Lemma~\ref{lemma:beta}, one concludes 
that ~$\alpha^\circ (\P) \geq A_{\rm INT}(\P)^{-1}\cdot 2^{-(k-1)}$, establishing 
the theorem. 

$\mbox{}$
\\
In view of the above, in order to prove Theorem~\ref{theorem:main}, 
it suffices to construct $Z \subseteq \Z^2$ such that $\F=\{\P + z\}_{z\in Z}$\, 
has Property~\ref{pr:crucial}. The remaining part of this section
is dedicated to constructing such $Z$, and proving that $\F$ 
has the required property.

The set of translates $Z$ is constructed as follows. Assume that $\P=\P_{INT}$\,. 
In particular, the vertices of $\P$ are in $\Z^2$.
For every direction $d \in D(\P)$, define the vector $v_d \in \Z^2$ as the 
difference between (\emph{any}) pair of two \emph{consecutive} 
integer lattice points on ({\em any}) line in $L_{d}$, the set of all lines in direction 
$d$ through an integer lattice point.

\begin{figure}[ht]
    \centering
    \includegraphics[width=7cm]{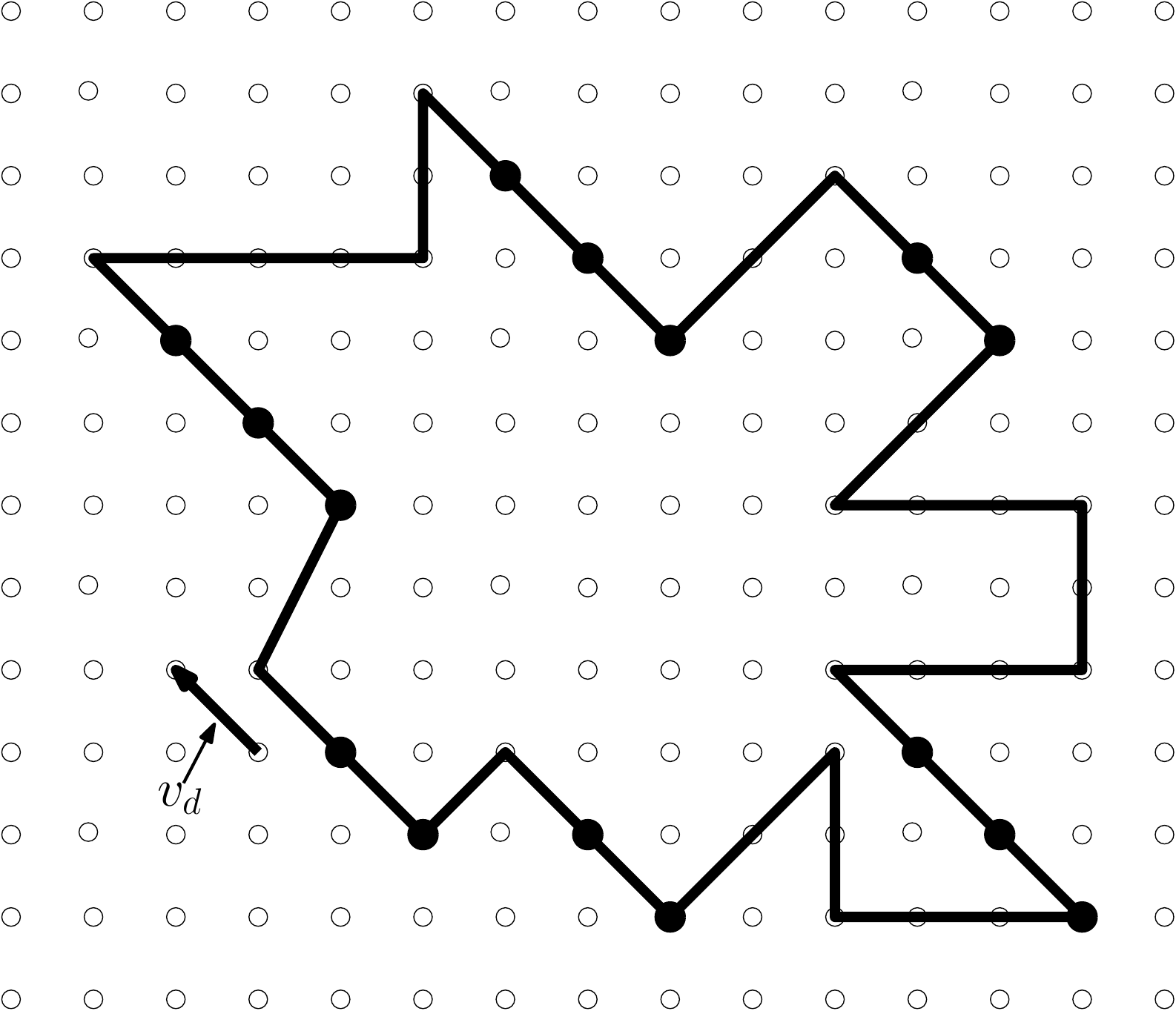}
        \caption{The points of $A_{d}$ are the filled discs in the picture}
        \label{figure:Ad}
\end{figure}

Let $A_{d}$ be the set of all integer lattice points $z$ 
on the boundary of $\P$ such that both 
$z$ and $z+v_{d}$ lie on an edge of $\P$ in the direction $d$ (see Figure
\ref{figure:Ad}). Let $\A = \{-A_d\}_{d\in D(\P)}$.
By Theorem~\ref{theorem:f}, there exists a $\Z_2$-weighting \,$\Fgoth$\, of $\Z^2$ that is stable,
but not $0$-stable, with respect to $\A$. 
Define $Z$ as the support of $\Fgoth$, i.e., $Z=\{z \mid \Fgoth(z)=1\}$.
Finally, define $\F = \{\P + z\}_{z\in Z}$. Our goal is to show that the family 
$\F=\{\P + z\}_{z\in Z}$ has Property~\ref{pr:crucial}.

Call a direction $d$ of an edge of $\P$ \emph{active} if
$\Fgoth(-A_{d})=1$, and \emph{passive} if $\Fgoth(-A_i)=0$. \\
We claim that a point $p \in \Z^2$ belongs to an odd number of sets in 
$\{A_d + z\}_{z\in Z}$ if $d$ is active, and to an even number of those sets 
if $d$ is passive.
Indeed, the number of solutions of the equation ~$a + z =p$, where 
$a\in A_d,\;z\in Z$,
is precisely the size of $(-A_d+p) \cap Z$, and hence its parity is 
$\Fgoth(-A_d+p)=\Fgoth(-A_d)$, as desired.

Let $I\subset \ell \in L_d$, for some $d$, be an edge in the arrangement
of lines $\bigcup_{\footnotesize{d\in D(\P)}}$\,.
Notice that $I$ cannot contain integer lattice points in its (relative) interior. 
There exists two consecutive integer lattice points $p$ and $q$ on $\ell$
such that $I$ is contained in the line segment $J = [p,q] \subset \ell$.
Observe that $q-p$ is either $v_{d}$ or $-v_{d}$; assume w.l.o.g., that
$q-p=v_{d}$.

We claim that the parity of the number of sets from $\F=\{\P + z\}_{z\in Z}$
whose boundary contains $J$ is odd if $d$ is active, and even if it is passive.
Indeed, $J$ is contained in the boundary of $(P+z)$, $z\in Z$, if and only if 
$(A_d + z)$ contains $p$. As we have already seen, the parity of the number of 
such sets is odd if and only if $d$ is active. In particular, it depends only on $d$, and not on $I$, 
as desired. Moreover, by Theorem~\ref{theorem:f}, there exists at least one
active direction $d$.  

This concludes the verification of Property~\ref{pr:crucial} for the constructed family
${\cal F}=\{\P + z\}_{z\in Z}$, which in turn concludes the proof of Theorem \ref{theorem:main}.
\end{proof}

Notice that the above proof makes no use of the connectivity of $\P$ nor of 
the connectivity of its boundary. 
Thus, as it has been already mentioned in the Introduction, Theorem~\ref{theorem:main}
applies equally well to any compact figure in $\R^2$ with non-empty interior, piecewise
linear boundary, and finite number of vertices, all of which are rational. 
%
%


\begin{thebibliography}{99}

\ignore{
\bibitem{Pak13}
I.~Pak, \emph{Lectures on discrete geometry and convex polyhedra},
Cambridge U. Press (to appear), available at
http://www.math.ucla.edu/~pak/geompol8.pdf.
}

\bibitem{P14}
{R.~Pinchasi},
\emph{Points covered an odd number of times by translates}, 
{\it Amer. Math. Monthly} {\bf 121} (2014), no. 7, 632--636.

\bibitem{OPP15}
{A.~Oren, I.~Pak, R.~Pinchasi,} 
\emph{On the odd area of planar sets},
{\it Discrete Comput. Geom.} {\bf 55} (2016), no. 3, 715--724. 

\bibitem{rogers}
{C.A.~Rogers},
\emph{Packings and coverings}, Cambridge Tracts in Mathematics and Mathematical Physics, No. 54,  Cambridge 1964.

\bibitem{TT09}
\emph{The International Mathematics Tournament of the Towns, Fall 2009}, 
available at
http://www.math.toronto.edu/oz/turgor/archives/TT2009F\_JAproblems.pdf.

\end{thebibliography}
\end{document}